\documentclass[a4paper,11pt,oneside]{amsart}
\usepackage{amsbsy,amssymb,pstricks,pst-node,mathrsfs,MnSymbol}
\usepackage{mathtools,graphicx}
\usepackage{subfigure}
\usepackage[shortlabels]{enumitem}
\usepackage[pagewise]{lineno}
\usepackage{tikz}
\usetikzlibrary{positioning}
\newtheorem{theorem}{Theorem}[section]

\newtheorem{lemma}[theorem]{Lemma}
\newtheorem{remark}[theorem]{Remark}

\theoremstyle{case}

\newtheorem{corollary}[theorem]{Corollary}
\newtheorem{example}[theorem]{Example}
\newtheorem{note}[theorem]{Note}
\theoremstyle{definition}
\newtheorem{definition}[theorem]{Definition}
\newtheorem{observation}[theorem]{Observation}

\newcommand{\covering}{%
	\mathrel{-\mkern-4mu}<%
}

\begin{document}
	
	\title{On the strong metric dimension of the complement of the zero-divisor graph of a lattice}
	\maketitle
	\markboth{Pravin Gadge and Vinayak Joshi}{On the strong metric dimension of the complement of the zero-divisor graph of a lattice}
	\begin{center}\begin{large}Pravin Gadge$^\text{a}$ and  Vinayak Joshi$^\text{b}$\end{large}\\\begin{small}\vskip.1in$^\text{a}$\emph{ GES's Shri Bhausaheb Vartak Arts, Commerce and Science College, Borivali - 400091, Maharashtra, India.}\\$^\text{b}$\emph{Department of Mathematics, Savitribai Phule Pune University, Pune - 411007, Maharashtra, India.\\E-mail: praving2390@gmail.com (P. Gadge),  vvjoshi@unipune.ac.in (V. Joshi) }\end{small}\end{center}\vskip.2in
	\begin{abstract} In this paper,  we compute the strong metric dimension of the complement of the zero-divisor graph of the blow-up of a Boolean lattice. Using these results, we calculate the strong metric dimension of the total graph,  the maximal graph, the intersection graph of ideals, the complement of the zero-divisor graph of a reduced ring, and the component graph of a vector space.
	\end{abstract}\vskip.2in
	\noindent\begin{Small}\textbf{Mathematics Subject Classification (2020)}:
		05C25, 06A07, 05C17, 13A70.    \end{Small}\vskip.2in
	\vskip.25in
	
	\baselineskip 17truept 
	\section{\bf Introduction}\label{intro}
	
	\par In \cite{Be}, Beck introduced the concept of associating a graph with a ring, focusing primarily on colorings. Later, Anderson and Livingston \cite{AL} introduced the zero-divisor graph of a commutative ring $R$, denoted by $\Gamma(R)$. In this graph, the vertices represent all nonzero zero-divisors of $R$, and two distinct vertices $x$ and $y$ are connected by an edge if and only if $xy = 0$. Many researchers have since examined the connection between the ring-theoretic properties of $R$ and the graph-theoretic characteristics of $\Gamma(R)$.                           
	
	\par The idea of metric dimension of a graph was introduced by Harary and Melter \cite{hara}. In 2004, Seb\"{o} and Tannier \cite{sebo} introduced a more refined parameter of a graph known as the strong metric dimension. Since then, numerous researchers have explored both the metric dimension and the strong metric dimension for various types of graphs, including Cayley graphs, trees, digraphs, and Cartesian product graphs (see \cite{gcl}, \cite{dk}, \cite{skb}, \cite{galai}).    
	
	\par  Determining the metric and strong metric dimensions of graphs is known to be NP-complete, which has led some algebraic graph theorists to focus on finding these parameters for graphs arising from algebraic and ordered structures (see \cite{abachi}, \cite{ip}, \cite{ma}).
	
	\par In this paper, we compute the strong metric dimension of the complement of the zero-divisor graph of a blow-up of a Boolean lattice. Additionally, we provide applications to the total graph, the maximal graph, the intersection graph of ideals in a commutative ring $R$, the complement of the zero-divisor graph of a reduced ring, and the nonzero component graph of a vector space.
	
	\section{Preliminaries}
	\par By $G = G(V, E)$, we refer to a simple, undirected graph $G$ with vertex set $V = V(G)$ and edge set $E = E(G)$. The \textit{complement} of a graph $G$, denoted as $G^c$, is a graph with the same vertex set as $G$, and two vertices are adjacent in $G^c$ if and only if they are not \mbox{adjacent in $G$.}
	
	\par Let $N(v)$ represent the set of all vertices adjacent to a vertex $v$ in $G$, and $N[v] = N(v) \cup \{v\}$. A set $S$ of vertices in a graph $G$ is called a \textit{vertex cover} if every edge in $G$ has at least one endpoint in $S$. The \textit{vertex cover number} of $G$, denoted by $\alpha(G)$, is the minimum cardinality of a vertex cover in $G$. An \textit{independent set} in a graph $G$ is a set of vertices where no two vertices are adjacent. The \textit{independence number} of $G$, denoted by $\beta(G)$, is the size of the largest independent set in $G$. 
	
	\par For a connected graph $G$, consider a subset $S = \{v_1, v_2, \dots, v_k\}$ of $V(G)$, and let $v \in V(G) \setminus S$. The \textit{metric representation} of $v$ with respect to $S$ is expressed as the $k$-vector 
	\[
	D(v \mid S) = (d(v, v_1), d(v, v_2), \dots, d(v, v_k)),
	\]
	where $d(v, v_i)$ denotes the distance between $v$ and $v_i$. If, for $S \subseteq V(G)$, the equality 
	\[
	D(u \mid S) = D(v \mid S)
	\]
	holds for every pair of vertices $u, v \in V(G) \setminus S$, implying that $u = v$, then $S$ is referred to as a \textit{resolving set} for $G$.
	
	The \textit{metric basis} of $G$ is a resolving set $S$ of minimum cardinality, and the number of vertices in such a set is defined as the \textit{metric dimension} of $G$, denoted by $\dim_M(G)$.
	
	\par In a connected graph $G$, a vertex $w$ is said to \textit{strongly resolve} two vertices $u$ and $v$ if there exists a shortest path from $u$ to $w$ that contains $v$, or a shortest path from $v$ to $w$ that contains $u$. A set $W$ of vertices is called a \textit{strong resolving set} for $G$ if every pair of vertices in $G$ is strongly resolved by at least one vertex in $W$. The smallest cardinality of a strong resolving set for $G$ is called the \textit{strong metric dimension} of $G$, denoted by $\text{sdim}_M(G)$.
	
	\indent Let $L$ be a lattice with $0$. Given an element $a \in  L$,  the	 \textit{annihilator} of $a$, denoted by $a^\perp$, is the set of elements $b \in L$ such that $a\wedge b = 0$. 	
	Let $(L, \leq)$ be a lattice. The \textit{dual} of $L$, denoted by $(L^{\partial}, \geq)$, is the lattice with the partial order $a \geq b$ in $L^{\partial}$ if and only if $a \leq b$ in $L$.
	
	\par Let $x$ and $y$ be elements of $L$. Then $y$ \textit{covers} $x$, written $x \covering y$, if $x < y$ and there is no element $z$ such that $x < z < y$. If $0 \covering x$, then $x$ is called an \textit{atom} of $L$. Moreover, $L$ is called \textit{atomic} if every nonzero element contains an atom.
	
	A \textit{chain} is a lattice in which any two elements are comparable. If $a$ and $b$ are incomparable elements of $L$, we denote this by $a \parallel b$.
	
	\indent A lattice $L$ is said to be \textit{bounded} if $L$ has both the least element $0$ and the greatest element $1$. An element $b$ of a bounded lattice $L$ is called a \textit{complement} of $a \in L$ if $a\wedge b = 0$ and $a\vee b = 1$. A \textit{pseudocomplement} of $a \in L$ is an element $b \in L$ such that $a\wedge b = 0$, and if $a\wedge x = 0$, then $x \leq b$. It is straightforward to verify that for any element $a$ in $L$, there is at most one pseudocomplement, denoted by $a^{\ast}$ if it exists.	
	A bounded lattice $L$ is called \textit{complemented} (respectively, \textit{pseudocomplemented}) if every element of $L$ has a complement (respectively, $a^{\ast}$ exists for every $a \in L$).

	A \textit{zero-divisor} of $
	L$ is defined as any element of the set 
	$
	Z(L) = \{a \in L \mid \text{there exists } 0\neq b \in L \text{ such that } a\wedge b = 0\}.
	$
	The \textit{zero-divisor graph} of $L$ is the graph $G(L)$, where the vertices are the elements of $Z^\ast(L) = Z(L) \setminus \{0\}$, and two distinct vertices $a$ and $b$ are adjacent if and only if $a\wedge b = 0$.
	
	A lattice $L$ is called a \textit{$0$-distributive lattice} if $a \wedge b = 0$ and $a \wedge c = 0$ imply $a \wedge (b \vee c) = 0$. Dually, we have the concept of a \textit{$1$-distributive lattice}.
	
	\section{\bf Strong Metric Dimension of the complement of the Zero-Divisor Graph of  blow-up of a Boolean Lattice}
	
	In this section,  we compute the strong metric dimension of the complement of the zero-divisor graph of a blow-up of a Boolean lattice. Also, we study the structure of $G^c(L^B)_{SR}$ and compute the $sdim_{M}(G^c(L^B))$ through $G^c(L^B)_{SR}$.

	It is known that the complement of the zero-divisor graph of a poset need not be connected in general. However, in the case of an atomic $0$-distributive poset, Devhare et al. \cite[Theorem 2.15]{sdj1} proved that $G^c(P)$ is connected. We quote this result when $L$ is an atomic $0$-distributive lattice.
	
	\begin{theorem}[{Devhare et al. \cite[Theorem 2.15]{sdj1}}]\label{star}
		Let $L$ be an atomic $0$-distributive lattice such that $Z^*(L)\neq \{0\}$. Then $G^c(L)$ has at most two components. Moreover, if $L$ has exactly two atoms, then $G^c(L)$ is a disjoint union of two complete graphs, and $G^c(L)$ is connected if and only if $L$ has at least three distinct atoms. Furthermore, if $G^c(L)$ is connected then diam$(G^c(L))=2$ .
	\end{theorem}
	
	By using this fact, we have the following result.
	
	\begin{lemma}\label{finites}
		\par Let $L$ be an atomic $0$-distributive lattice with at least three atoms. Then $dim_M (G^c(L))$ is finite if and only if $G^c(L)$ is finite. 
	\end{lemma}  
	\begin{proof}
		Assume that  $dim_M (G^c(L))$ is finite. Let $W$ be the metric basis for $G^c(L)$ with $|W|=k$ for some non-negative integer $k$. By Theorem \ref{star}, the diameter of $G^c(L)$ is equal to $2$, i.e., $d(x,y)\in \{1,2\}$ for any distinct $x,y\in V(G^c(L))$. Then for each $x\in V(G^(L))$, the metric representation $D(x|W)$ is the $k$-coordinate vector, where each coordinate is in the set $\{1,2\}$. Thus, there are only $2^k$ possibilities for $D(x|W)$. Since $D(x|W)$ is unique for each $x\in V(G^c(L))$, so $|V(G^c(L))|\leq 2^k$. This implies that $V(G^c(L))$ is finite. Hence $G^c(L)$ is finite. The converse is obvious.		
	\end{proof}
	It is easily seen that every strong resolving set is also a resolving set, which leads to $\operatorname{dim}_M(G) \leq \operatorname{sdim}_M(G)$.

	Hence, we have the following corollary.
	\begin{corollary}
		Let $L$ be an atomic $0$-distributive lattice  with at least three atoms. Then $\operatorname{sdim}_M(G^c(L))$ is finite if and only if $G^c(L)$ is finite.
	\end{corollary} 
	
	The following result is due to Gallai, which states the relationship between the independence number $\beta(G)$ and the vertex cover number $\alpha(G)$ of a graph $G$. 
	
	\begin{theorem}[Gallai's Theorem] \label{galai}
		For any graph $G$ of order $n$, $\alpha(G) + \beta(G) = n$.
	\end{theorem}
	
	\begin{definition}
		A vertex $u$ of $G$ is {\it maximally distant} from $v$, if
		$d(v,w)\leq d(u,v)$ for every $w\in N(u)$. If $u$ is maximally distant from $v$ and $v$ is maximally distant from $u$, we describe  $u$ and $v$ as mutually maximally distant in $G$.
	\end{definition} 
	
	One can easily observe that if $u$ is maximally distant from $v$, then  $v$ need not be maximally distant from $u$. In a graph $G^c(L)$ shown in Figure \ref{figure1}, the vertex $(1,0,0)$ is maximally distant from $(1,1,0)$, however $(1,1,0)$ is not mutually maximally distant from $(1,0,0)$.\\
	The {\it boundary} of $G$ is defined as: 
	
	$\partial(G)=\{u\in V(G) \mid \text{there is}~ v\in V(G)~ \text{such that} ~$u,v$~ \text{are mutually maximally distant}\}$.
	
	The $\partial(G)$ of the graph $G^c(L)$ is given in Example \ref{example1}.
	
	We use the notion of a strong resolving graph introduced by Oellermann and Peters-Fransen in  \cite{galai}.
	\begin{definition} [Oellermann and Peters-Fransen \cite{galai}]
		Let $G$ be a graph. The \textit{strong resolving graph} of $G$ is denoted by $G_{SR}$, with the vertex set $V(G_{SR})$ equal to $\partial(G)$ and two distinct vertices $u$ and $v$ are connected in $G_{SR}$ if and only if $u$ and $v$ are mutually maximally distant in $G$.
	\end{definition}
	\begin{note}
		Let $K_t$ denotes the complete graph on $t$ vertices. Then $(K_t)_{SR}=K_t$.
	\end{note}
		\begin{lemma}\label{cjoin}
		Let $L$ be an atomic $0$-distributive lattice with at least three atoms. Then $(G^c(L)\vee K_t)_{SR}=G^c(L)_{SR}+ (K_t)_{SR}=G^c(L)_{SR}+ K_t$, for some $t\in \mathbb{N}$.
	\end{lemma}
	\begin{proof}
		First, we show that $V((G^c(L)\vee K_t)_{SR})=V(G^c(L)\vee K_t)$. To prove this, it is sufficient to show that $V(G^c(L)\vee K_t)\subseteq V((G^c(L)\vee K_t)_{SR})$. 
		
		Let $x\in V(G^c(L)\vee K_t)=V(G^c(L))\cup V(K_t)$. 
		
		Suppose $x\in V(G^c(L))$.  Then there exists some $y\in V(G^c(L))$ such that $x\wedge y=0$. This implies that $d(x,y)=2=diam(G^c(L)\vee K_t)$, by Theorem  \ref{star}. This means that $x$ and $y$ are mutually maximally distant in $G^c(L)\vee K_t$. Thus $x\in V((G^c(L)\vee K_t)_{SR})$. 
		
		Now, if $x\in V(K_t)$, then for every $y\in V(K_t)$, we have $d(y,z_1)\leq d(x,y)=1$ and  $d(x,z_2)\leq d(x,y)=1$ for every $z_1\in N(x)$ and $z_2\in N(y)$. This means that $x$ and $y$ are mutually maximally distant in $G^c(L)\vee K_t$ and  hence $x\in V((G^c(L)\vee K_t)_{SR})$.

Therefore $	V(G^c(L) \vee K_t) \subseteq V\left((G^c(L) \vee K_t)_{{SR}}\right)$.
		Thus,
		\begin{equation}\label{equation}
			V\left((G^c(L) \vee K_t)_{{SR}}\right) = V(G^c(L) \vee K_t) = V(G^c(L)) \cup V(K_t)
		\end{equation}		
		
		Now, we prove that $V(G^c(L)_{SR})=V(G^c(L))$. To prove this it is sufficient to show that $V(G^c(L))\subseteq V(G^c(L)_{SR})$. 
		
		Let $x\in V(G^c(L))$. Then there exist some $y\in V(G^c(L))$ such that $x\wedge y=0$. This implies that $d(x,y)=2=diam (G^c(L))$. This means that $x$ and $y$ are mutually maximally distant in $G^c(L)$ and hence $x\in V(G^c(L)_{SR})$. Thus $V(G^c(L)_{SR})=V(G^c(L))$.	
		
		Now, $V(G^c(L)_{SR}+K_t)=V(G^c(L)_{SR})\cup V(K_t)=V(G^c(L))\cup V(K_t)$. From equation (1), we have $V((G^c(L)\vee K_t)_{SR})=V(G^c(L)_{SR}+K_t)$.

		Let $ x $ be adjacent to $ y $ in $ (G^c(L) \vee K_t)_{SR} $, i.e., $ x $ is mutually maximally distant with $ y $ in $ G^c(L) \vee K_t $.

		We claim that $ x $ is adjacent to $ y $ in $ (G^c(L)_{SR} + K_t) $, i.e., we need to prove that $ x $ is adjacent to $ y $ in $ G^c(L)_{SR}$ or $x$ is adjacent to $y$ in $K_t $. Since $(K_t)_{SR}=K_t$, we essentially prove that either $ x $ is mutually maximally distant with $ y $ in $ G^c(L) $, or $ x $ is mutually maximally distant with $ y $ in $ K_t$.

		 To prove this, it is sufficient to prove that either $x,y\in V(G^c(L))$ or $x,y\in V(K_t)$. 
		 
		On the contrary assume that $ x \in V(G^c(L)) $ and $ y \in V(K_t) $. Then there exists some $ z \in V(G^c(L)) $ such that $ x\wedge z = 0 $, i.e., $ d(x,z) = 2 $. But then in $ G^c(L) \vee K_t $,   $ z\in N(y) $ and $ d(x,z) = 2 \nleq 1 = d(x,y) $, a contradiction to $ x $ and $ y $ being mutually maximally distant in $ G^c(L) \vee K_t $.		
		Thus we  have either $ x, y \in V(G^c(L)) $ or $ x, y \in V(K_t) $, and in both the cases, $ x $ is mutually maximally distant with $y$ either in $G^c(L)$ or in $K_t$. Thus, we have proved that whenever $x$ is adjacent to $ y $ in $ (G^c(L) \vee K_t)_{SR} $, then $x$ is adjacent to $y$ in $G^c(L)_{SR} + K_t$.

		Now, assume that $ x $ is adjacent to $ y $ in $ G^c(L)_{SR} + K_t $.  
		This means that either $x$ is adjacent to $y$ in $ G^c(L)_{SR} $ or $ x $ is adjacent to $ y $ in $ K_t $.
		
		Suppose $ x $ is adjacent to $ y $ in $ K_t $, i.e., $ x $ is mutually maximally distant with $ y $ in $ K_t $.  
		We claim that $ x $ is adjacent to $ y $ in $ (G^c(L) \vee K_t)_{SR} $.  
		
		As $ x, y \in V(K_t) $, then we have $d(y,z_1)\leq d(x,y)=1$ and  $d(x,z_2)\leq d(x,y)=1$ for every $z_1\in N(x)$ and $z_2\in N(y)$.  
		Thus, $ x $ is mutually maximally distant with $ y $ in $ G^c(L) \vee K_t $.  
		Hence $ x $ is adjacent to $ y $ in $ (G^c(L) \vee K_t)_{SR} $.
		
		Now suppose $ x $ is adjacent to $ y $ in $ G^c(L)_{SR} $.  
		Then $ x $ is mutually maximally distant with $ y $ in $ G^c(L) $.  
		We prove that $ x $ is mutually maximally distant with $ y $ in $ G^c(L) \vee K_t $.
		
		As $ x $ is mutually maximally distant with $ y $ in $ G^c(L) $, then for all $ t_1 \in N(x) $, we have $d(t_1, y) \leq d(x, y)$ and
		for all $ t_2 \in N(y), \ d(t_2, x) \leq d(x, y)$		for some $ t_1, t_2 \in V(G^c(L)) $.
		
		Now in $ G^c(L) \vee K_t $, for every $t_3\in V(K_t)$, we have $t_3\in N(x)$  as well as $t_3 \in N(y)$ with $
		d(x, t_3) = d(y, t_3)$.
		This shows that $ x $ and $ y $ are mutually maximally distant in $ G^c(L) \vee K_t $.
		Hence $ x $ is adjacent to $ y $ in $ (G^c(L) \vee K_t)_{SR} $.

		Thus $(G^c(L) \vee K_t)_{SR}=G^c(L)_{SR}+(K_t)_{SR}=G^c(L)_{SR}+K_t$.	
	\end{proof}

	To compute the strong metric dimension of a connected graph $G$, it is enough to find the vertex cover number of $G_{SR}$; see \cite[Theorem 2.1]{galai}. Therefore, this section is devoted to studying the structure of the strong resolving graph of $G^c(L^B)$, where $L^B$ is a blow-up of a Boolean lattice $L\cong \mathbf{2}^n$.

	\begin{theorem}[{Oellermann and Peters-Fransen \cite[Theorem 2.1]{galai}}] \label{gala}
		For any connected graph $G$, $sdim_M(G) = \alpha(G_{SR})$.
	\end{theorem}
	
	\begin{example}\label{example1}
		
		Let $L=C_2\times C_2\times C_2$ and  $G^c(L)$ be the complement of the zero-divisor graph of $L$. From Figure \ref{figure1}, we can easily see that $W=\{(1,0,0),(0,1,0),(0,0,1)\}$ is the only minimum cardinality strong resolving set and hence   $sdim_{M}(G^c(L))=3$. On the other hand, we have $\partial(G^c(L))=V(G^c(L))$ and $\beta(G^c(L)_{SR})=3$. Therefore $\alpha(G^c(L)_{SR})=6-\beta(G^c(L)_{SR})=6-3=3$, this also gives  $sdim_M(G^c(L)) = 3$. Note that the strong metric dimension of a graph $G$, which is  isomorphic to $G^c(L)$ is calculated (i.e., $sdim_{M}(G^c(L))=3$) in Example 2.5 of \cite{abachi}.
		
		\begin{figure}[h]

			\begin{center}
				\begin{tikzpicture}	[scale=1]		
					\begin{scope}[shift={(-1,0)}]
						\draw [fill=black] (0,0) circle (0.05);
						\draw [fill=black] (-1,1) circle (0.05);
						\draw [fill=black] (0,1) circle (0.05);
						\draw [fill=black] (1,1) circle (0.05);
						\draw [fill=black] (-1,2) circle (0.05);
						\draw [fill=black] (0,2) circle (0.05);
						\draw [fill=black] (1,2) circle (0.05);
						\draw [fill=black] (0,3) circle (0.05);

						\draw (0,0) -- (-1,1) -- (-1,2) -- (0,3);
						\draw (0,0) -- (0,1) -- (-1,2) ;
						\draw (0,0) --  (1,1) -- (1,2) -- (0,3);
						\draw (-1,1) -- (0,2) -- (0,3);
						\draw (0,1) -- (1,2);
						\draw (1,1) -- (0,2);
						\node at (0,-0.9) {\tiny  $ L=C_2\times C_2 \times
							C_2$};
						
						\node at (0,-0.3) {\tiny $(0,0,0)$};
						\node at (-1.6,1) {\tiny $(1,0,0)$};
						\node at (-1.6,2) {\tiny $(1,1,0)$};
						\node at (0,3.2) {\tiny $(1,1,1)$};
						\node at (0,1.4) {\tiny $(0,1,0)$};
						\node at (1.6,1) {\tiny $(0,0,1)$};
						\node at (1.6,2) {\tiny $(0,1,1)$};
						\node at (0.5,2.2) {\tiny $(1,0,1)$};

					\end{scope}

					\begin{scope}[shift={(4,0)}]
						\draw [fill=black] (0,0) circle (0.05);
						\draw [fill=black] (-2,0) circle (0.05);
						\draw [fill=black] (2,0) circle (0.05);
						\draw [fill=black] (1,1.5) circle (0.05);
						\draw [fill=black] (-1,1.5) circle (0.05);
						\draw [fill=black] (0,3) circle (0.05);

						\draw (0,0)--(2,0)--(1,1.5)--(0,3)--(-1,1.5)--(-2,0)--(0,0);
						\draw (-1,1.5)--(1,1.5)--(0,0)--(-1,1.5);
						
						\node at (0,-1.0) {\tiny  $G^c(L)$};
						
						\node at (0,-0.3) {\tiny $(1,0,1)$};
						\node at (-2.55,0) {\tiny $(1,0,0)$};
						\node at (-1.55,1.5) {\tiny $(1,1,0)$};
						\node at (0,3.2) {\tiny $(0,1,0)$};
						\node at (1.6,1.5) {\tiny $(0,1,1)$};
						\node at (2.55,0) {\tiny $(0,0,1)$};
						
					\end{scope}
					
				\end{tikzpicture}
				\caption{A lattice $L$ and its $ G^c(L)$}\label{figure1}
			\end{center}
		\end{figure}
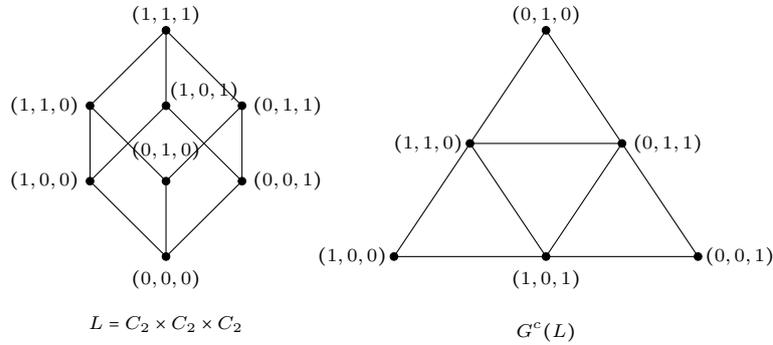
	\end{example}
	
	The blow-up of a graph was first introduced by M. Ye et al. in \cite{mytwl}. The generalized blow-up of a Boolean lattice using finite chains introduced by Gadge and Joshi in \cite{pgvj}.
	
	\begin{definition}[{Gadge and Joshi \cite[Definition 2.11]{pgvj}}]\label{blowup}
		The blow-up $L^B$ of a Boolean lattice $L\cong \mathbf{2}^n$ using chains
		is   obtained as follows:
		\begin{enumerate}
			\item For $1\leq i\leq n$, replace each atom $q_i$ of $L$ by a chain $C_i$ of finite  length, say $m_i-1$, with elements $q_i=x_i^{1},x_i^{2},\dots,x_i^{m}$ such that  $x_{i}^1 \covering x_{i}^2 \covering \dots \covering x_{i}^{m_i}  $. 
			\item Let $ x=\bigvee\limits_{j=1}^{k}q_{i_j}\in L\setminus\{1\}$, where $q_{i_j}$ be atoms of $L$ with $i_j\in\{1,2,\dots, n\}$. \linebreak Replace $x\in L$ by a chain 		
			$C_{i_1i_2 \dots i_k}$ of  finite length, say  $n_{j}-1$,  with elements $x=x_{i_1i_2 \dots i_k}^{1}, ~ x_{i_1i_2 \dots i_k}^{2},~ \dots, x_{i_1i_2 \dots i_k}^{n_j}$ for some $n_{j}\in \mathbb{N}$  such that 
			$x_{i_1i_2 \dots i_k}^{1}\covering x_{i_1i_2 \dots i_k}^{2}\covering \dots \covering x_{i_1i_2 \dots i_k}^{n_{j}}$,
			where $\{i_1,i_2,\dots, i_k\} \subseteq \{1,2,\dots, n\} $.
			\item The elements $0$ and  $1$ of $L$ will be represented by  $\mathbf{0}$ and $\mathbf{1}$ in $L^B$ respectively.

		\end{enumerate}
	\end{definition}
	We will represent the elements of $L^B$ as tuples as follows.\\
	An element $x^{t}_{i_1i_2\dots i_k}$  ($1\leq t \leq n_j$ for some $n_j\in \mathbb{N}$) on the chain $C_{i_1i_2\dots i_k}$ ($\{i_1,i_2,\dots i_k\}\subseteq\{1,2,\dots,n\}$) can be represented by the tuples $(z_{1},z_{2},\dots, z_{n})$ where 
	$$
	z_i=\begin{cases}
		t & \text{if $i\in \{i_1,i_2,\dots , i_k\}$ }\\
		0 & \text{otherwise}.
	\end{cases}
	$$
	
	The blow-up $L^B$ of $L\cong \mathbf{2}^3$ is shown in following Figure \ref{figure3}. 
	
	Define a relation $\sim$ on $L^B$ as $x\sim y$ if and only if $x^\perp=y^\perp$. $\sim$ is an equivalence relation on $L^B$. Let  $[a]$ denotes the equivalence class of $a$ under $\sim$. The set of equivalence classes of $L^B$ will be denoted by $[L^B]$=\{$[a]\mid a \in L^B$\}. It is easy to verify that $[L^B]$ is a lattice under the partial order $[x] \leq [y]$ if and only if $y^\perp \subseteq x^\perp$. Also, in the blow-up lattice $L^B$, if we take any two elements, say $a, b$, on any chain which is obtained by replacing an element in the Boolean lattice $L\cong \mathbf{2}^n$, then $a^\perp = b^\perp$. Hence $[a]=[b]$.

	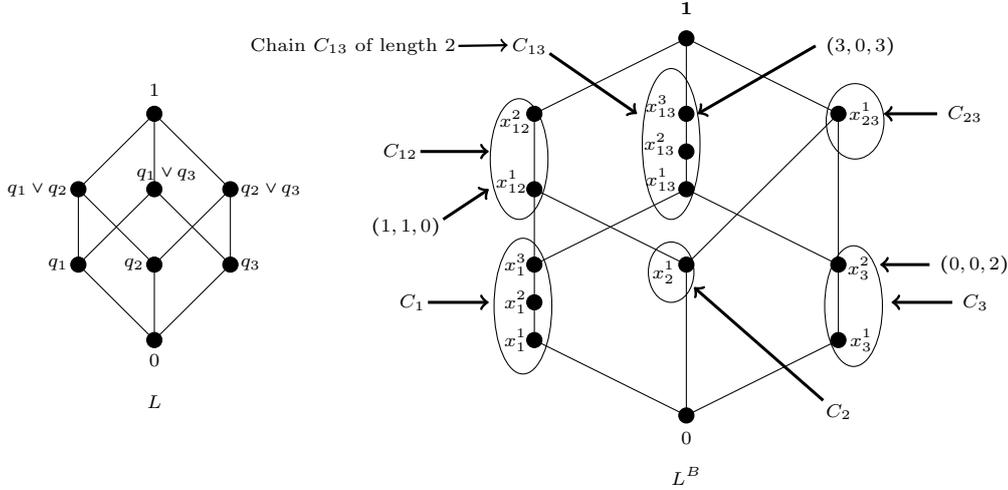
\begin{figure}[h]
		\begin{center}
			\begin{tikzpicture}	[scale=1]		
				
				\begin{scope}[shift={(-2,1)}]
					\draw (0,0) -- (-1,1); \draw (0,0) -- (0,1); \draw (0,0) -- (1,1);
					\draw (-1,1) -- (-1,2); \draw (-1,1) -- (0,2); \draw (0,1) --
					(-1,2); \draw (0,1) -- (1,2); \draw (1,1) -- (1,2); \draw (1,1) --
					(0,2);
					\draw (-1,2) -- (0,3); \draw (0,2) -- (0,3); \draw (1,2) -- (0,3);
					\draw [fill=black] (0,0) circle (.1); \draw [fill=black] (0,1)
					circle (.1);\draw [fill=black] (0,2) circle (.1);\draw
					[fill=black] (0,3) circle (.1);
					\draw [fill=black] (-1,1) circle (.1); \draw [fill=black] (-1,2)
					circle (.1);
					\draw [fill=black] (1,1) circle (.1);\draw [fill=black] (1,2)
					circle (.1);
					\node [below] at (0,-.05) {\tiny$0$};\node [left] at (-1,1)
					{\tiny$q_1$};\node [right] at (1,1) {\tiny$q_3$}; \node [left] at (0,1)
					{\tiny$q_2$};
					\node [left] at (-1,2) {\tiny$q_1\vee q_2$}; \node [left] at (0.7,2.2) {\tiny$q_1\vee q_3$};
					\node [right] at (1,2) {\tiny$q_2\vee q_3$};
					\node [above] at (0,3.1) {\tiny$1$};
					\node[below] at (0,-0.6) {\tiny$L$};
					
				\end{scope}
				
				\begin{scope}[shift={(5,0)}]
					
					\draw [fill=black] (0,0) circle (.1);
					\draw [fill=black] (-2,1) circle (.1);
					
					\draw [fill=black] (-2,2) circle (.1);
					
					\draw [fill=black] (-2,3) circle (.1);
					
					\draw [fill=black] (-2,4) circle (.1);

					\draw [fill=black] (0,2) circle (.1);
					\draw [fill=black] (0,4) circle (.1);
					\draw [fill=black] (-2,1.5) circle (.1);
					\draw [fill=black] (0,3.5) circle (.1);
					\draw [fill=black] (0,3) circle (.1);
					
					\draw [fill=black] (2,2) circle (.1);
					\draw [fill=black] (2,4) circle (.1);
					\draw [fill=black] (2,1) circle (.1);
					\draw [fill=black] (0,5) circle (.1);
					
					\draw (0,0) -- (-2,1) --  (-2,2) -- (-2,3) -- (-2,4)--(0,5)-- (0,4) --(0,3) --(-2,2);
					\draw (0,0) -- (0,2) --(-2,3);
					\draw (0,0) -- (2,1)--(2,2) --(2,4) --(0,5);
					\draw (0,2) -- (2,4);
					\draw (0,3) -- (2,2);

					\draw (-2.15,1.45) ellipse (0.38 and 0.9); \draw (-2.2,3.4) ellipse (0.38 and 0.8); \draw (2.2,1.45) ellipse (0.38 and 0.8); \draw (-0.2,1.9) ellipse (0.3 and 0.4); \draw (-0.2,3.6) ellipse (0.38 and 1);\draw (2.22,3.9) ellipse (0.38 and 0.5);
					
					\draw [line width=0.40mm,->] (-3.4,1.5)--(-2.6,1.5);
					\node [left] at (-3.3,1.5) {\tiny $C_{1}$};
					
					\draw [line width=0.40mm,->] (-3.5,3.5)--(-2.6,3.5);
					\node [left] at (-3.4,3.5) {\tiny $C_{12}$};
					
					\draw [line width=0.40mm,<-] (2.7,1.5)--(3.5,1.5);
					\node [right] at (3.5,1.5) {\tiny $C_{3}$};
					
					\draw [line width=0.40mm,<-] (2.6,4)--(3.3,4);
					\node [right] at (3.3,4) { \tiny$C_{23}$};
					
					\draw [line width=0.40mm,<-] (-0.65,4)--(-1.8,4.8);
					\node [left] at (-1.7,4.9) {\tiny$C_{13}$};
					
					\draw [line width=0.30mm,->] (-3,4.9)--(-2.35,4.91);
					\node [left] at (-2.9,4.9) {\tiny Chain $C_{13}$ of length $2$};

					\draw [line width=0.40mm,<-] (0.1,1.7)--(1.8,0.2);
					\node [right] at (1.7,0.05) {\tiny $C_{2}$};
					
					\draw [line width=0.40mm,<-] (2.55,2)--(3.2,2);
					\node [right] at (3.2,2) { \tiny$(0,0,2)$};

					\draw [line width=0.40mm,<-] (0.15,4)--(1.7,4.9);
					\node [right] at (1.7,4.9) { \tiny$(3,0,3)$};
					
					\draw [line width=0.40mm,->] (-3.2,2.6)--(-2.6,3);
					\node [left] at (-3.1,2.5) { \tiny$(1,1,0)$};

					\node at (-2.25,1) {\tiny$ x_1 ^{1}$};
					\node at (-2.25,1.5) {\tiny$ x_1 ^{2}$};
					\node at (-2.25,2) {\tiny$ x_1 ^{3}$};
					\node at (-0.3,1.9) {\tiny$ x_2 ^{1}$};
					\node at (2.27,1.95) {\tiny$ x_3 ^{2}$};
					\node at (-2.3,3.1) {\tiny$ x_{12} ^{1}$};
					\node at (-2.25,3.9) {\tiny$ x_{12} ^{2}$};
					\node at (-0.33,4.1) {\tiny$ x_{13} ^{3}$};
					\node at (-0.35,3.6) {\tiny$ x_{13} ^{2}$};
					\node at (-0.33,3.13) {\tiny$ x_{13} ^{1}$};
					\node at (2.35,4) {\tiny$ x_{23} ^{1}$};
					\node at (2.3,1) {\tiny$ x_3 ^{1}$};
					
					\node at (0,-0.8) {\tiny$ L^B$};

					\node at (0,-0.3) {\tiny 0};
					\node at (0,5.4) {\tiny $\mathbf{1}$};
				\end{scope}

			\end{tikzpicture}
			\caption{Boolean lattice $L\cong \mathbf{2}^3$ and its blow-up $L^B$ }\label{figure3}
		\end{center}
	\end{figure}

	\textbf{Throughout this paper, $L^B$ denotes the blow-up of a Boolean lattice $L\cong \mathbf{2}^n $ with $n\geq 3$.}
	\begin{remark}[{Gadge and Joshi \cite[Remark 2.13]{pgvj}}]\label{pcbool}
		Note that if $a, b \in L$ ($a\not=b$), where $L$ is Boolean and $C_a$ and $C_b$ are the corresponding chains in $L^B$, then $a\wedge b =x \wedge y $ and $a\vee b =x \vee y $ in $L^B$  for any element $x$ on the chain $C_a$ and any element $y$ on the chain $C_b$. Hence, in particular, if $x\in L^B$ and $x^*$ be the pseudocomplement of $x$ in $L^B$, then in $L^B$, we have $x \vee x^*=1$ and $x \wedge x^* =0$. Note that $x^*$ need not be the unique complement of $x$ in $L^B$, whereas $x^*$ is the unique complement of $x$ in $L$. Also, in $L^B$, the pseudocomplement of atom $q_i$ is the dual atom of $L^B$, denoted by $q_i^*$, and the pseudocomplement of dual atom $q_i^*$ is the largest element in the chain of $[q_i]$. 
	\end{remark}

	\begin{corollary}[{Gadge and Joshi \cite[Corollary 2.15]{pgvj}}]\label{pseudo}
		
		Let $L^B$ be a blow-up of a Boolean lattice $L\cong \mathbf{2}^n$. Then the following statements  hold: \begin{enumerate}
			\item $L^B$ and its dual lattice $(L^B)^{\partial}$ both are pseudocomplemented.
			\item $[L^B]\cong [(L^B)^\partial]\cong L\cong\mathbf{2}^n$.
		\end{enumerate}

	\end{corollary}
	\begin{remark}\label{rem4.6}
		Since every pseudocomplemented lattice is 0-distributive and the converse is true if a lattice is finite, we have, by Corollary \ref{pseudo} and by Definition \ref{blowup}, $L^B$ is an atomic $0$-distributive lattice. Hence by Theorem \ref{star}, the $diam(G^c(L^B))=2$.
	\end{remark}
	The following result \cite[Theorem 2.16]{pgvj} gives the relation between zero-divisor graph of a $0-$distributive lattice and blow-up of a Boolean lattice $L\cong \mathbf{2}^n$. 
	\begin{theorem}[{Gadge and Joshi \cite[Theorem 2.16]{pgvj}}]\label{0-disblowup}
		Let $L'$ be a finite 0-distributive lattice with $n$ atoms. Let $L^B$ be the blow-up of the Boolean lattice $L=\mathbf{2}^n$. Then $G(L')=G(L^B)$.
	\end{theorem}

	Let $x$,$y\in V(G^c(L^B))$ with $x\sim y$. Then $a\wedge b\neq 0$, i.e., $a$ and $b$ are adjacent in $G^c(L^B)$ for every $a\in [x]$ and $b\in [y]$. Note that if $x$ and $y$ are vertices of $G^c(L^B)$ with $[x]=[y]$, then $N(x)=N(y)$ in $G^c(L^B)$.

	The following result establishes the relation between $G^c(L^B)$ and $G^c(L^B)_{SR}$.
	
	\begin{lemma}\label{compblow}
		Suppose that $L^B$ is a blow-up of a Boolean lattice $L\cong \mathbf{2}^n$. Then the following statements hold:
		\begin{enumerate}
			\item $V(G^c(L^B)_{SR})=V(G^c(L^B))$.
			
			\item For every $x$, $y\in V(G^c(L^B)_{SR})$, if $[x]=[y]$, then $x$ is adjacent to $y$ in $G^c(L^B)_{SR}$.
			\item  For every $x$, $y\in V(G^c(L^B)_{SR})$, if $[x]\neq[y]$, then $x$ is adjacent to $y$ in $G^c(L^B)_{SR}$ if and only if $x$ is not adjacent to $y$ in $G^c(L^B)$.
			
		\end{enumerate}
	\end{lemma}	
	\begin{proof}
		\begin{enumerate}
			
			\item It is sufficient to show that $V(G^c(L^B))\subseteq V(G^c(L^B)_{SR})$. Assume that $a\in V(G^c(L^B))$. Then there exists some $b\in V(G^c(L^B))=V(G(L^B))$ such that $a\wedge b =0$. This means that $d_{G^c(L^B)}(a,b)=2=diam (G^c(L^B))$ by Theorem \ref{star}. This implies that   $a$ is mutually maximally distant from $b$. Thus $a\in V(G^c(L^B)_{SR})$.
			
			\item Let $a,b\in V(G^c(L^B)_{SR})$ with $[a] =[b] $. Then $N_{G^c(L^B)}(a)=N_{G^c(L^B)}(b)$.  Hence $a$ and $b$ are mutually maximally distant, that is,  $a$ is adjacent to $b$ in $G^c(L^B)_{SR}$.
			
			\item Let $a,b\in V(G^c(L^B)_{SR})$ and $[a] \neq [b]$. Suppose  $a$ and $b$ are not adjacent  in $G^c(L^B)$, i.e., $a\wedge b=0$. Then by Theorem \ref{star}, $d_{G^c(L^B)}(a,b)=2=diam(G^c(L^B))$. Hence $a$ and $b$ are mutually maximally distant, that is,  $a$ is adjacent to $b$ in $G^c(L^B)_{SR}$.
			
			To prove the converse, we assume contrary that $a$ and $b$ are mutually maximally distant and $a$ is adjacent to $b$ in $G^c(L^B)$ for some $a, b\in V(G^c(L^B))$. Hence $a\wedge b \neq 0$ and hence $d_{G^c(L^B)}(a,b)=1$.
			
			We claim that either $a^\ast \wedge b \neq 0$ or $a\wedge b^\ast \neq 0$. Assume, on the contrary, that  $a^\ast \wedge b = 0$ and  $a \wedge b^\ast =0$. Then $a^\ast \leq b^\ast$, $b^\ast \leq a^\ast$, and hence $[a]=[b]$, a contradiction. Hence without loss of generality, assume that $a^\ast \wedge b \neq 0$, i.e., $a^\ast \in N_{G^c(L^B)}(b)$. Thus 
			$d_{G^c(L^B)}(a^\ast,a)=2>1=d_{G^c(L^B)}(a,b)$, a contradiction to  $a$ and $b$ are mutually maximally distant in $G^c(L^B)$. 			
			Hence $a$ is adjacent to $b$ in $G^c(L^B)$.
		\end{enumerate}
	\end{proof}

	In view of Lemma \ref{compblow}, we have:
	
	\begin{corollary}\label{corol1}
		Let $L^B$ be a blow-up of a Boolean lattice $L\cong \mathbf{2}^n$. If $L^B\cong L\cong \mathbf{2}^n$, then $G^c(L^B)_{SR}=G(L^B)$.
	\end{corollary}
	
	In the following Figure \ref{figure5}, $G^c(L)$ is the complement of the zero-divisor graph of a lattice given in Figure \ref{figure1}. Figure \ref{figure5} illustrates   Lemma \ref{compblow} as well as  Corollary \ref{corol1}. 
	
	\begin{figure}[h]
		
		\begin{center}
			\begin{tikzpicture}	[scale=1]		
				\begin{scope}[shift={(-1,0)}]
					\draw [fill=black] (0,0) circle (0.05);
					\draw [fill=black] (-2,0) circle (0.05);
					\draw [fill=black] (2,0) circle (0.05);
					\draw [fill=black] (1,1.5) circle (0.05);
					\draw [fill=black] (-1,1.5) circle (0.05);
					\draw [fill=black] (0,3) circle (0.05);

					\draw (0,0)--(2,0)--(1,1.5)--(0,3)--(-1,1.5)--(-2,0)--(0,0);
					\draw (-1,1.5)--(1,1.5)--(0,0)--(-1,1.5);
					
					\node at (0,-1.0) {\tiny  $G^c(L)$};
					
					\node at (0,-0.3) {\tiny $(1,0,1)$};
					\node at (-2.55,0) {\tiny $(1,0,0)$};
					\node at (-1.65,1.5) {\tiny $(1,1,0)$};
					\node at (0,3.2) {\tiny $(0,1,0)$};
					\node at (1.6,1.5) {\tiny $(0,1,1)$};
					\node at (2.55,0) {\tiny $(0,0,1)$};
					
				\end{scope}
				
				\begin{scope}[shift={(1.5,0)}]
					
					\draw [fill=black] (4,3) circle (0.05);
					\draw [fill=black] (5,0.5) circle (0.05);
					\draw [fill=black] (3,0.5) circle (0.05);
					\draw [fill=black] (4,2) circle (0.05);
					\draw [fill=black] (6,0) circle (0.05);
					\draw [fill=black] (2,0) circle (0.05);
					\draw (3,0.5) -- (4,2) -- (4,2)  -- (5,0.5);
					\draw (2,0)--(3,0.5);
					\draw (3,0.5) -- (5,0.5);
					\draw (6,0)--(5,0.5);
					\draw (4,3)--(4,2);
					\node at (3.3,0.15) {\tiny $(1,0,0)$};
					\node at (4,3.2) {\tiny $(1,0,1)$};
					\node at (5.65,0.6) {\tiny $(0,0,1)$};
					\node at (6.6,-0.15
					) {\tiny $(1,1,0)$};
					\node at (2,-0.25) {\tiny $(0,1,1)$};
					\node at (3.45,2.1) {\tiny $(0,1,0)$};
					\node at (4,-0.9) {\tiny $G^c(L)_{SR}\cong G(L)$};
				\end{scope}
			\end{tikzpicture}
			\caption{$G^c(L)$ and its $G^c(L)_{SR}\cong G(L)$ }\label{figure5}
		\end{center}
	\end{figure}
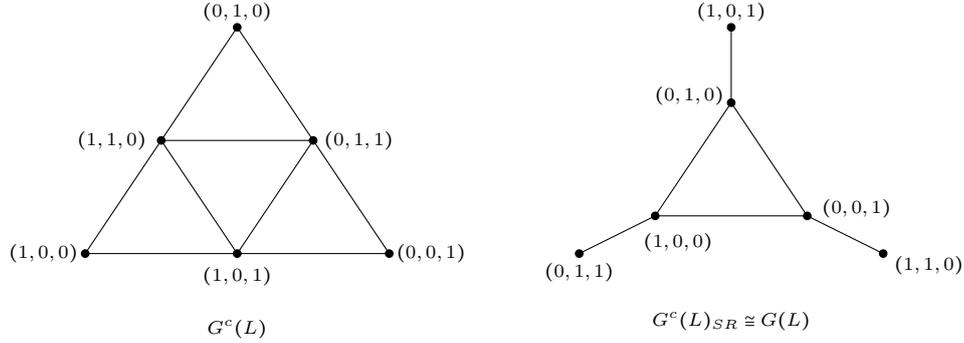
	
	\begin{lemma}\label{comconnect}
		Let $L^B$ be a blow-up of a Boolean lattice $L\cong \mathbf{2}^n$. Then $G^c(L^B)_{SR}$ is a connected graph with $diam(G^c(L^B)_{SR})$ is at most $3$. 
	\end{lemma}
	
	\begin{proof}
		
		By Lemma \ref{compblow} (1), we have $V(G^c(L^B))=V(G^c(L^B)_{SR})$.

		
		Using the representation of elements of $L^B$, we assume that $a=(a_1,a_2,\dots,a_n)$,\\ $b=(b_1,b_2,\dots,b_n)\in V(G^c(L^B))$.  We have to find a path between $a$ and $b$ in $G^c(L^B)$.
		To see this, consider the following two cases:
		
		\vspace{3mm}
		
		\textbf{Case 1.}   $[a] =[b] $.
		\par By Lemma \ref{compblow} (2), we have $a$ and $b$ are adjacent in $G^c(L^B)_{SR}$.
		\vspace{3mm}

		\textbf{Case 2.}
		$[a] \neq [b] $.
		\par If $a$ is not adjacent to $b$ in $G^c(L^B)$, then by Lemma \ref{compblow} (3), $a$ is adjacent to $b$ in $G^c(L^B)_{SR}$.
		\vspace{3mm}

		\par Now, suppose that $a$ is  adjacent to $b$ in $G^c(L^B)$, i.e., $a\wedge b \neq 0$. \\
		We consider two subcases:
		
		\textbf{Subcase 2(1).} $a$ and $b$ are comparable.\\
		Let $a\leq b$ or $b\leq a$. Then for some $i$, $1\leq i\leq n$ such that $a_i=b_i=0$, as $a,b\in V(G^c(L^B)_{SR})$. Let $c=(0,\dots,0, c_i,0,\dots, 0)$ with $c_i\neq 0$ such that $a\wedge c=0$ and $b\wedge c=0$. This implies that $d_{G^c(L^B)}(a,c)=d_{G^c(L^B)}(b,c)=2$. Hence, $a$ and $c$ are mutually maximally distant, and $b$ and $c$ are also mutually maximally distant. This means that $a$ and  $c$ are adjacent  and $c$ and $b$ are adjacent in $G^c(L^B)_{SR}$.  Therefore, we have a path $ a - c - b$.
		\vspace{3mm}

		\textbf{Subcase 2(2).} $a$ and $b$ are incomparable.\\
		Now, assume that $a\nleq b$ and $b\nleq a$. Then there exists $i$ and $j$ such that  $1\leq i\neq j\leq n$ such that $a_i\neq 0 $, $b_j\neq 0$ and $a_j=0$, $b_i=0$. Let $c=(0,\dots,0, c_j,0,\dots, 0)$ and $d=(0,\dots,0, d_i,0,\dots, 0)$ with $c_j\neq 0$ and $d_i\neq 0$. In this case, $a\wedge c=0$, $c\wedge d=0$ and $d\wedge b=0$. Hence by Theorem \ref{star} $d_{G^c(L^B)}(a,c)=d_{G^c(L^B)}(c,d)=d_{G^c(L^B)}(d,b)=diam(G^c(L^B))=2$. This means that $a$ and $c$, $c$ and $d$, $d$ and $b$ are mutually maximally distant. This means we have a path  $a-c-d-b$. 
		This shows that $G^c(L^B)_{SR}$ is a connected graph and  $diam(G^c(L^B)_{SR})\leq 3$.	\end{proof}

	Consider the lattice $L^B$ given in Figure \ref{figure3}. The graphs $G(L^B)$, $G^c(L^B)$ and $G^c(L^B)_{SR}$ of $L^B$ are shown in Figure \ref{figure6} and Figure \ref{figure4}, respectively. This illustrates  Lemma \ref{comconnect}.
	\begin{figure}[h]
		\begin{center}
			\begin{tikzpicture}	[scale=1]

				\begin{scope}[shift={(-4,0)}]
					\draw [fill=black] (-2,1) circle (.05);
					\draw [fill=black] (-1,2) circle (.05);
					
					\draw [fill=black] (1,2) circle (.05);

					\draw [fill=black] (-0.7,1.5) circle (.05);
					
					\draw [fill=black] (-0.3,1.1) circle (.05);
					
					\draw [fill=black] (1.9,0.9) circle (.05);
					\draw [fill=black] (2.4,1.6) circle (.05);
					\draw [fill=black] (-1,4) circle (.05);
					\draw [fill=black] (1,4) circle (.05);
					\draw [fill=black] (0,4) circle (.05);
					\draw [fill=black] (0,3) circle (.05);
					\draw [fill=black] (1,1.5) circle (.05);
					
					\node at (0.4,3.051) {$ x_2 ^{1}$};
					\node at (1.2,2.2) {\bf$ x_3 ^{2}$};
					\node at (-0.15,0.85) {$ x_1 ^{1}$};
					\node at (-0.75,1.25) {\tiny$ x_1 ^{2}$};
					\node at (-1.2,2.3) {$ x_1 ^{3}$};
					\node at (2.0,0.61) {$ x_{12} ^{1}$};
					\node at (2.55,1.32) {\bf$ x_{12} ^{2}$};
					\node at (-1,4.3) {\bf$ x_{13} ^{1}$};
					\node at (1,4.3) {\bf$ x_{13} ^{3}$};
					\node at (-2,0.65) {\bf$ x_{23} ^{1}$};
					\node at (0,4.3) {\bf$ x_{13} ^{2}$};
					\node at (1,1.1) {\bf$ x_3 ^{1}$};
					
					\draw (-2,1)--(-0.3,1.1) -- (1,2);
					\draw (-0.3,1.1) -- (1,1.5)--(2.4,1.6);
					\draw (1,1.5)--(1.9,0.9);
					\draw (1,1.5)--(0,3);
					\draw (-0.7,1.5)--(-2,1);
					\draw (-0.7,1.5) -- (1,1.5);
					\draw (1,1.5) -- (-1,2);
					\draw (-0.7,1.5) -- (1,2);
					\draw (-0.7,1.5) -- (0,3)--(-0.3,1.1);
					\draw (0,3)--(1,4);
					\draw (0,3)--(0,4);
					\draw (0,3)--(-1,4);
					\draw (0,3)-- (-1,2)-- (1,2)--(0,3);
					\draw (-2,1)--(-1,2);
					\draw (1,2)--(1.9,0.9);
					\draw (1,2)--(2.4,1.6);
					
				\end{scope}
			\end{tikzpicture}
			\caption{$G(L^B)$}\label{figure6}
		\end{center}
	\end{figure} 
	
	\begin{figure}[h]
		
		\begin{center}
			\begin{tikzpicture}	[scale=0.6]

				\begin{scope}[shift={(-10,0)}]
					
					\draw [fill=black] (-5,0) circle (.05); 
					\draw [fill=black] (-4.45,2.3) circle (.05);
					\draw [fill=black] (0,5) circle (.05);
					\draw [fill=black] (-2.75,4.2) circle (.05);
					\draw [fill=black] (4.45,2.3) circle (.05);
					\draw [fill=black] (2.75,4.2) circle (.05);
					\draw [fill=black] (4.45,-2.3) circle (.05);
					\draw [fill=black] (0,-5) circle (.05);
					\draw [fill=black] (-4.45,-2.3) circle (.05);
					\draw [fill=black] (-2.75,-4.2) circle (.05); 
					\draw [fill=black] (2.75,-4.2) circle (.05);		
					\draw [fill=black] (5,0) circle (.05);
					
					\draw (-2.75,4.2)--(2.75,4.2)--(0,5)--(2.75,-4.2)--(-5,0)--(2.75,4.2)--(2.75,-4.2)--(-2.75,4.2)--(0,5)--(4.45,2.3)--(-4.45,2.3)--(2.75,-4.2)--(4.45,2.3)--(-2.75,4.2)--(-4.45,2.3)--(-2.75,-4.2)--(4.45,2.3)--(0,-5)--(-4.45,2.3)--(0,5)--(-2.75,-4.2)--(0,-5)--(-2.75,4.2)--(-2.75,-4.2)--(2.75,4.2)--(0,-5)--(2.75,-4.2)--(-2.75,-4.2);
					
					\draw (0,-5)--(0,5);
					\draw(-5,0)--(4.45,-2.3)--(-2.75,4.2);
					\draw(4.45,-2.3)--(2.75,4.2)--(-4.45,2.3)--(5,0);
					\draw(4.45,-2.3)--(-4.45,2.3);
					\draw(4.45,-2.3)--(4.45,2.3);
					\draw(4.45,-2.3)--(2.75,-4.2)--(-4.45,-2.3)--(-5,0)--(-2.75,4.2)--(-4.45,-2.3)--(4.45,-2.3)--(5,0)--(4.45,2.3)--(2.75,4.2)--(-4.45,-2.3);
					
					\node at (5.35,0) {$x_2^1$};
					\node at (0,5.35) {$x_{1}^3$};
					\node at (0,-5.35) {$x_1^2$};
					\node at (-5.35,0) {$x_3^1$};
					\node at (-4.95,2.35) {$x_{12}^2$};
					\node at (4.95,2.35) {$x_{12}^1$};
					\node at (-4.85,-2.35) {$x_{3}^2$};
					\node at (4.95,-2.35) {$x_{23}^1$};
					\node at (-3.05,4.8) {$x_{13}^1$};
					\node at (3.15,4.50) {$x_{13}^3$};
					\node at (-3.10,-4.45) {$x_1^1$};
					\node at (3.05,-4.55) {$x_{13}^2$};
					
					\node at (0,-7.45) {$G^c(L^B)$};
					
				\end{scope}
				
				\begin{scope}[shift={(2,0)}]
					
					\draw [fill=black] (-5,0) circle (.05); 
					\draw [fill=black] (-4.45,2.3) circle (.05);
					\draw [fill=black] (0,5) circle (.05);
					\draw [fill=black] (-2.75,4.2) circle (.05);
					\draw [fill=black] (4.45,2.3) circle (.05);
					\draw [fill=black] (2.75,4.2) circle (.05);
					\draw [fill=black] (4.45,-2.3) circle (.05);
					\draw [fill=black] (0,-5) circle (.05);
					\draw [fill=black] (-4.45,-2.3) circle (.05);
					\draw [fill=black] (-2.75,-4.2) circle (.05); 
					\draw [fill=black] (2.75,-4.2) circle (.05);		
					\draw [fill=black] (5,0) circle (.05);
					
					\draw(-2.75,-4.2)--(0,-5)--(0,5)--(-2.75,-4.2)--(4.45,-2.3)--(0,5)--(-5,0)--(0,-5)--(-4.45,-2.3)--(0,5)--(5,0)--(0,-5)--(4.45,-2.3);
					
					\draw(-2.75,-4.2)--(-4.45,-2.3);
					\draw(-2.75,-4.2)--(-5,0);
					\draw(-2.75,-4.2)--(5,0);
					\draw(-2.75,4.2)--(5,0)--(2.75,4.2)--(-2.75,4.2)--(2.75,-4.2)--(5,0);
					\draw(2.75,-4.2)--(2.75,4.2);
					\draw(-4.45,2.3)--(4.45,2.3)--(-5,0)--(-4.45,2.3)--(-4.45,-2.3)--(4.45,2.3);
					\draw(-5,0)--(-4.45,-2.3);
					\draw(-5,0)--(5,0)--(-4.45,-2.3);

					\node at (5.35,0) {$x_2^1$};
					\node at (0,5.35) {$x_{1}^3$};
					\node at (0,-5.35) {$x_1^2$};
					\node at (-5.35,0) {$x_3^1$};
					\node at (-4.85,2.35) {$x_{12}^2$};
					\node at (5,2.35) {$x_{12}^1$};
					\node at (-4.85,-2.35) {$x_{3}^2$};
					\node at (4.85,-2.35) {$x_{23}^1$};
					\node at (-3.00,4.65) {$x_{13}^1$};
					\node at (3.05,4.45) {$x_{13}^3$};
					\node at (-3.05,-4.45) {$x_1^1$};
					\node at (3.05,-4.45) {$x_{13}^2$};
					
					\node at (0,-7.45) {$G^c(L^B)_{SR}$};
					
				\end{scope}         
				
			\end{tikzpicture}
			
			\caption{$G^c(L^B)$ and $G^c(L^B)_{SR}$  }\label{figure4}
		\end{center}
	\end{figure}
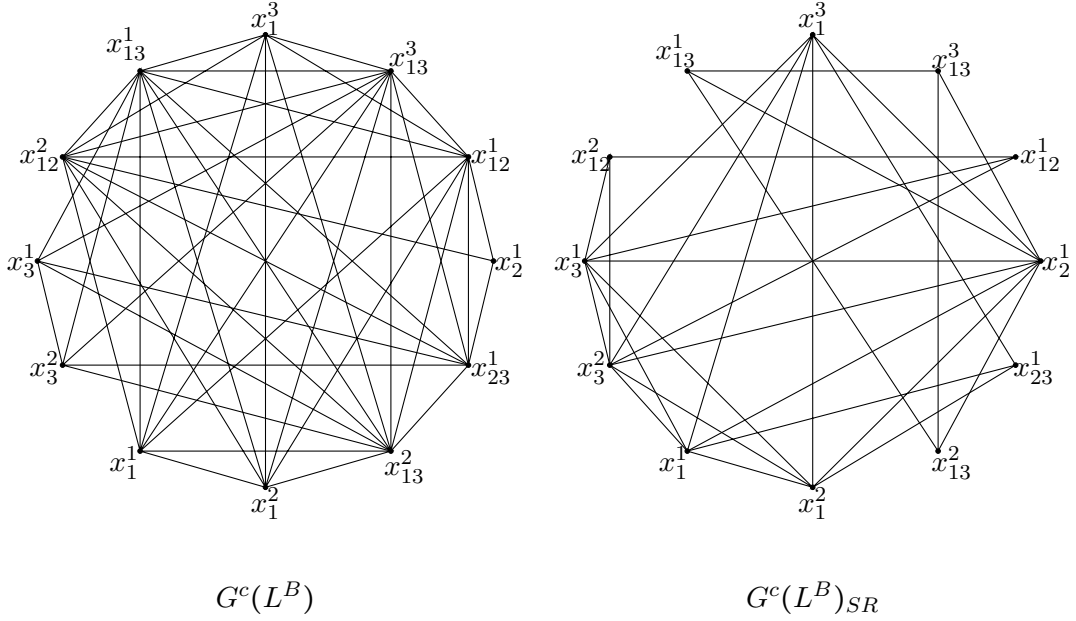
	\newpage
	\begin{lemma}\label{compinde}
		Let $L^B$ be a blow-up of a Boolean lattice $L\cong \mathbf{2}^n$. Then $\beta(G^c(L^B)_{SR})=2^{n-1}-1$.
	\end{lemma}
	
	\begin{proof}
		By Lemma \ref{comconnect},  $G^{c}(L^B)_{SR}$ is  connected. We prove that $\beta(G^{c}(L^B)_{SR})=2^{n-1}-1$.

		By Lemma~\ref{compblow}, if $[x] = [y]$, then $x$ and $y$ are adjacent in $G^{c}(L^B)_{SR}$, for  $x, y \in V(G^{c}(L^B)_{SR})$. This implies that at most one element from each equivalence class $[x]$ can be included in an independent set of $G^{c}(L^B)_{SR}$ for every $x \in V(G^{c}(L^B)_{SR})$. 
		
		Therefore, computing the largest independent set of $G^{c}(L^B)_{SR}$ is equivalent to computing the largest independent set of $G^{c}([L^B])_{SR}$. By Corollary~\ref{pseudo}~(2), this is the same as computing the largest independent set of $G^{c}(\mathbf{2}^n)_{SR}$.

		Now, we claim that $\beta(G^{c}(L)_{SR})=2^{n-1}-1$, where $L\cong \mathbf{2}^n$.

		For this, we define the following subsets of vertices:		  
		\begin{align*}
			V_1 &= \{(1, x_2, x_3, \dots, x_n) \mid x_i \in \{0,1\},\ \text{for } 2 \leq i \leq n\}, \\
			V_2 &= \{(0,1, x_3, \dots, x_n) \mid x_i \in \{0,1\},\ \text{for } 3 \leq i \leq n\}, \\
			&\ \ \vdots \\
			V_{n-1} &= \{(0, 0, \dots, 1, x_n) \mid x_n \in \{0,1\} \}, \\
			V_n &= \{(0, 0, \dots, 1)\}.
		\end{align*}
		
		We can easily observe that:
		$$
		1 = |V_n| < |V_{n-1}| < \dots < |V_1| = 2^{n-1} - 1 = \frac{2^n - 2}{2} = \frac{1}{2} \left| V(G^c(L)_{SR}) \right|.
		$$
		
		Since every vertex in $V(G^c(L)_{SR}) \setminus V_1$ is adjacent to at least one vertex in $V_1$, and for any $x, y \in V_1$, we have $x \wedge y \neq 0$, it follows from Lemma~\ref{compblow}~(3) that no two vertices in $V_1$ are adjacent.
		
		Therefore, $|V_1|$ is the maximum size of a set of vertices in $G^c(L)_{SR}$ such that no two are adjacent, and every vertex outside $V_1$ is adjacent to at least one vertex in $V_1$.
		
		Thus, we conclude	$
		\beta\left(G^{c}(L \cong \mathbf{2}^n)_{SR}\right) = 2^{n-1} - 1 \quad \text{and hence} \quad \beta\left(G^{c}(L^B)_{SR}\right) = 2^{n-1} - 1.
		$
	\end{proof}
	
	To illustrate  Lemma \ref{compinde}, consider the lattice $L^B$ depicted in Figure \ref{figure3} and the corresponding graphs $G^c(L^B)$ and $G^c(L^B)_{SR}$ of $L^B$ shown in  Figure \ref{figure4}. The largest independent set is $V_1=\{(1,0,0),(1,0,1),(1,1,0)\}=\{x_{1}^1,x_{12}^{1},x_{13}^1\}$ and $|V_1|=3$. Thus $\beta(G^{c}(L^B)_{SR})=3$.

	Now, we state the main result of this section.
	
	\begin{theorem}\label{compthm}
		Let $L^B$ be a blow-up of a Boolean lattice $L\cong \mathbf{2}^n$. Then $sdim_{M}(G^c(L^B))=|Z^*(L^B)|-2^{n-1}+1$.
	\end{theorem}
	\begin{proof}
		By Lemma \ref{compinde}, we have  $\beta(G^{c}(L^B)_{SR})=2^{n-1}-1$. Then by Theorem \ref{galai}, we have $\alpha(G^{c}(L^B)_{SR})+\beta (G^{c}(L^B)_{SR})=|V(G^{c}(L^B))_{SR}|$. Then by Theorem \ref{gala}, we have $sdim_{M}(G^{c}(L^B))=\alpha(G^{c}(L^B)_{SR})=|V(G^{c}(L^B))_{SR}|-\beta (G^{c}(L^B)_{SR})$. Thus, by Lemma \ref{compblow}, $sdim_{M}(G^{c}(L^B))=|V(G^{c}(L^B)_{SR})|-\beta(G^{c}(L^B)_{SR})=|V(G^{c}(L^B))|-2^{n-1}+1=|Z^*(L^B)|-2^{n-1}+1$. 
	\end{proof}
	Now, we close this section with the following corollary, which follows from  Theorem \ref{compthm}.
	\begin{corollary}
		Let $L\cong \mathbf{2}^n$ be a  Boolean lattice. Then  $sdim_{M}(G^c(L))=2^{n} -2^{n-1}-1$.
	\end{corollary}
	\begin{proof}
		The proof follows from Theorem \ref{compthm} and  Corollary \ref{corol1}. 
	\end{proof}
	
	\section{Applications to graphs from Algebraic Structures}\label{applications}
	In this section, we provide some applications of our results to the total graph, the maximal graph of a commutative ring, complement of zero-divisor graph of a reduced ring the intersection graph of ideals,  and the nonzero component graph of vector spaces.
	
	\subsection{Total graph of nonzero annihilating ideals of a ring} 
	Let $R$ be a commutative ring with identity, and let $Id(R)$ denote the set of all ideals of $R$. It is well known that 
	$Id(R)$ forms a complete modular, $1$-distributive lattice under inclusion as a partial order, with $(0)$ as the least element and $R$ as the greatest element. Also, sup$\{I,J\}=I+J$ and  inf$\{ I, J\} = I\cap J$.

	We denote the lattice $\text{Id}(R)$ by $L$ and let $L^\partial$ represent its dual. In $L^\partial$, the supremum and infimum of $\{I, J\}$ is $I \cap J$, and  $I + J$ respectively. The ideal $R$ serves as the least element in $L^\partial$, while the ideal $(0)$ is the greatest element. By duality, $L^\partial$ is a 0-distributive lattice. Additionally, the maximal ideals of $R$ are precisely the atoms of $L^\partial$, which makes $L^\partial$ an atomic lattice. Thus $L^\partial$ is an atomic $0-$distributive lattice.
	
	In \cite{svd}, Visweswaran and Patel introduced the concept of the annihilating ideal graph (also referred as the \textit{total graph} of nonzero annihilating ideals or \textit{sum annihilating ideal graph}) for a commutative ring $R$ with identity, denoted by $\Omega(R)$. The vertex set of $\Omega(R)$ consists of all nonzero annihilating ideals of $R$, and two distinct vertices 
	$I$ and $J$ are adjacent if and only if the sum $I+J$ is an annihilating ideal of
	$R$.
	
	\begin{definition}[{Ye and Wu \cite{mytw}}] Let $R$ be a commutative ring with identity. The {\it comaximal ideal graph}, $\mathbb{CG}(R)$ is a simple graph with its vertices are the nonzero proper ideals of $R$ not contained in Jacobson radical $J(R)$ of $R$ and two distinct vertices $I$ and $J$ are adjacent if and only if $I+J=R$.
	\end{definition}
	
	\begin{definition}[{Akbari et al. \cite {saaa}}] Let $R$ be a commutative ring with identity. The {\it co-annihilating ideal graph} of $R$, denoted by $\mathbb{CAG}(R)$ is a graph whose vertex set is the set of all non-zero proper ideals of $R$ and two distinct vertices $I$ and $J$ are adjacent whenever $ Ann(I) \cap  Ann (J)=\{0\}$, where $Ann(I)=\{x\in R ~|~ xi=0~ \text{for all } i \in I\}$.
	\end{definition}
	
	The following observation is due to Gadge and Joshi \cite{pjtotal}.
	\begin{observation}[{Gadge and Joshi \cite[Observation 4.6 (2)]{nkvj23}}]\label{observ}
		If $R$ is an Artinian ring, then $\Omega(R)=\mathbb{CAG}(R)^{c}=\mathbb{CG}(R)^{c} $.
	\end{observation}
	
	In \cite{mytwl}, M. Ye et al. proved that the comaximal ideal graph $\mathbb{CG}(R)$ is the blow-up of the zero-divisor graph of a Boolean lattice $\mathbf{2}^n$. In fact, they proved,
	
	\begin{theorem}[{M. Ye et al. \cite[Theorem 3.1]{mytwl}}]\label{comaxiid}
		Let $R$ be a ring with $|\text{Max}(R)|=n$, where $2 \leq n < \infty$. Then $\mathbb{CG}(R)$ is a blow-up of the zero-divisor graph of a Boolean lattice $\mathbf{2}^n$.
	\end{theorem}
	
	\begin{theorem}[{Khandekar and Joshi \cite[Theorem 5.1]{nkvj23}}]\label{comaxi}
		Let $R$ be a commutative ring with identity and let $Id(R)^\partial=L^\partial$ be the dual of the lattice $Id (R)$ of all ideals of $R$. Then $\mathbb{CG}(R)^c=G^c(Id(R)^\partial)=G^c(L^\partial)$.
	\end{theorem}
	
	By Theorem \ref{compthm}, Theorem \ref{comaxiid},  and Theorem \ref{comaxi}, we have the following results.
	\begin{corollary} [{N. Abachi et al. \cite[Theorem 2.7]{abachi}}]
		Suppose $R$ is a reduced commutative ring with identity. If $dim_{M}(\Omega(R))$ is finite  and  $|\text{Max(R)}|=n\geq 3$, then  $sdim_{M}(\Omega(R))=sdim_{M}(\mathbb{CAG}(R)^{c})=sdim_{M}(\mathbb{CG}(R)^{c})=   
		2^n-2^{n-1}-1$.
	\end{corollary}
	
	\begin{corollary} [{N. Abachi et al. \cite[Theorem 3.2]{abachi}}]
		Suppose that $R\cong R_1 \times R_2 \times \dots \times R_n$, where $R_i$ is an Artinian local ring such that $|A(R_i)^{*}|\geq 1$, for $1\leq i\leq n$. Then $sdim_{M}(\Omega(R))=sdim_{M}(\mathbb{CAG}(R)^{c})=sdim_{M}(\mathbb{CG}(R)^{c})=|V(\Omega(R))|-2^{n-1}+1$.
	\end{corollary}
	
	\subsection{Maximal graph of a ring}
	\indent In \cite{pdsb}, Sharma and Bhatwadekar  introduced a graph $\Gamma_{0}(R)$ associated with a commutative ring $R$ with identity, where the vertices represent the elements of $R$. Two distinct vertices $x$ and $y$ are adjacent if and only if  $Rx + Ry=R$. 
	
	Maimani et al. \cite{hm} investigated the subgraphs $\Gamma_1(R)$, $\Gamma_2(R)$, and $\Gamma_2^{\prime}(R) = \Gamma_2(R) \setminus J(R)$. Here, $\Gamma_1(R)$ is the subgraph of $\Gamma(R)$ induced on the set of units of $R$, $\Gamma_2(R)$ is the subgraph induced on the set of non-units of $R$, and $\Gamma_2^{\prime}(R)$ is the subgraph induced on the set of non-units of $R$ that are not in the Jacobson radical $J(R)$, i.e., $\Gamma_2^{\prime}(R) = \Gamma(R) \setminus(U(R) \cup J(R))$.

	Moconja and Petrovi\'{c} \cite{smm} demonstrated that comaximal graphs are blow-ups of Boolean graphs, which are the zero-divisor graphs of Boolean rings, or equivalently, Boolean lattices. However, the explicit construction of a Boolean lattice was not provided. The following result is essentially proven in \cite{pgg}.
	
	\begin{theorem}[{Gadge et al. \cite[Theorem 3.16]{pgg}}]\label{comax}
		Let $R$ be a finite commutative ring with identity such that $|Max(R)|=n$. Then $\Gamma_2'(R)^c= G^c({L}^B)$, where $L^B$ is the blow-up of a Boolean lattice $L\cong \mathbf{2}^n$. 
	\end{theorem}
	
	In \cite{asag}, Gaur and Sharma studied the maximal graph of a ring $R$, where the vertices represent the elements of $R$, and two distinct vertices $x$ and $y$ are adjacent if and only if both are contained in a maximal ideal of $R$. They noted that the maximal graph of a ring $R$ is equivalent to the $\Gamma_{2}'(R)^c$, complement of the comaximal graph of $R$.
	
	The following result follows from  Theorem \ref{comax} and Theorem \ref{compthm}. 
	
	\begin{theorem}
		Let $\Gamma_{2}'(R)^c$ be the maximal graph of a commutative ring $R$ with identity and $|\mathrm{Max}(R)|=n$, $n\geq 3$. Then $sdim_{M}(\Gamma_{2}'(R)^c)=|V(\Gamma_{2}'(R)^c)|-2^{n-1}+1$.
	\end{theorem}
	
	\subsection{Complement of the Zero-divisor graph of a reduced ring}

	Now, we compute the strong metric dimension of the complement of the zero-divisor graph of a reduced ring.

	\begin{theorem}[{\cite[Remark 3.4]{jlkr}}, {\cite[Lemma 3.3]{sdj1}}]\label{zero}
		Let $\Gamma(R)$ be the ring-theoretic zero-divisor graph of a finite reduced commutative ring  $R$ with identity. Then $\Gamma(R)$ equals to the lattice-theoretic zero-divisor graph of $G(\prod_{i=1}^{n} C_i) $, where $C_i$'s are the chains with $|C_i|=|F_i|$, where $R=\prod_{i=1}^{n}F_i$ ($F_i$'s are finite fields.).
	\end{theorem}
	
	The following result follows from Theorem \ref{compthm} and Theorem \ref{zero}.
	\begin{theorem}\label{reduced}
		Let  $R \cong \prod_1^n \mathrm{~F}_i$, where $F_i $ is a field for every $1 \leq i \leq n$, then $\operatorname{sdim}_M(\Gamma^c(R))=$ $\left|Z^*(R)\right|-2^{n-1}+1$.
		
	\end{theorem}

	\subsection{Intersection graph of a ring}
	Let $R$ be a commutative ring with identity.
	The \textit{intersection graph of ideals}, denoted by $\mathbb{IG}(R)$, is defined for a commutative ring $R$ with identity. The vertex set of this graph consists of all nonzero proper ideals of $R$, where two distinct vertices $I$ and $J$ are adjacent if and only if  $I \cap J \neq {0}$  (see \cite{cgms2009}).
	
	Now, let us consider the zero-divisor graph $G^*(Id(R))$ of the lattice of ideals $Id(R)$. Clearly, the vertex set of $G^*(Id(R))$ is  $Id(R) \setminus \{0_{Id(R)}, 1_{Id(R)}\}$, which represents the set of all nonzero proper ideals of $R$. As a result, we have the equality $V(\mathbb{IG}(R)) = V(G^*(Id(R)))=V(G^{*c}(Id(R)))$.
	
	In $G^*(Id(R))$, two vertices $I$ and $J$ are adjacent if and only if $I \wedge J = 0_{Id(R)}$, which is equivalent to $I \cap J = \{0\}$. In contrast, in $G^{*c}(Id(R))$, the ideals $I$ and $J$ are adjacent  whenever $I \cap J \neq \{0\}$. Therefore, the following result holds.
	
	\begin{lemma}[{Khandekar and Joshi \cite[Lemma 5.5]{nkvj23}}]\label{ideal} Let $R$ be a  commutative ring with identity. Then  $\mathbb{IG}(R)={G}^{*c}(Id(R))$.
	\end{lemma}

	The following discussion can be found in \cite{djl2019}.
	
	\medskip
	
	\noindent
	Recall that a commutative ring $R$ with identity is called a \emph{special principal ideal ring} (abbreviated as \emph{SPIR}) if it is a local Artinian principal ideal ring, as defined in \cite{th1968}. Let $M$ be a maximal ideal of an SPIR $R$. Then there exists a natural number $n$ such that $M^n = \{0\}$, $M^{n-1} \neq \{0\}$, and every ideal $I$ of $R$ is of the form $I = M^i$ for some $i \in \{0,1,\dots,n\}$ \cite[Proposition~4]{th1968}. Thus, $M$ is a nilpotent ideal with nilpotency index $n$. Moreover, the lattice of ideals of $R$ is isomorphic to a chain $C_{n+1}$ of length $n$ (i.e., a totally ordered set with $n+1$ distinct elements).
	
	\medskip
	
	\noindent
	According to \cite[Lemma~10]{th1968}, a commutative ring $R$ is an Artinian principal ideal ring if and only if there exist SPIRs $R_1, \dots, R_n$ such that
	$
	R \cong R_1 \times \cdots \times R_n,
	$
	as also stated in the structure theorem for Artinian rings \cite[Theorem~8.7]{am1969}. Therefore, for an Artinian principal ideal ring $R$, the lattice of ideals $\mathrm{Id}(R)$ is isomorphic to the product of chains, denoted by
	$	\mathbf{C} = \prod_{i=1}^n C_{n_i+1}$,
		where each $C_{n_i+1}$ is a chain of length $n_i$, corresponding to the nilpotency index of the maximal ideal $M_i$ in $R_i$.

	Let $R$ be a commutative ring  with finitely many ideals. Suppose $	R \cong R_1 \times \cdots \times R_n$, where $R_i$ is a PIR non-field. We denote the number of proper ideals of $R_i$ by $|I(R_i)|$. Let $|I(R_i)|=n_i$ for every $1\leq i\leq n$ and $n\geq 2$ is a positive integer.  Assume that $I=I_1 \times \cdots \times I_n$ and $J=J_1 \times \cdots \times J_n$ are vertices of $\mathbb{IG}(R)$, where $I_i, J_i \in R_i$, for every $1 \leq i \leq n$. Let $NZC(I)$ denotes the number of zero components of  $I$. 
	
	We denote $A_i=\{I\in V(\mathbb{IG}(R))\mid NZC(I)=i\}$, for $0\leq i\leq n-1$ and $V(\mathbb{IG}(R))=V(G^{*c}(Id(R)))=\bigl(\bigcup\limits_{i=0}^{n-1} A_i\bigr)\setminus R$. Also, note that for every $I,J\in A_0\setminus 
	R$, we have $I\wedge J\neq 0_{Id(R)}$ and $|A_0|=\prod_{i=1}^{n}(n_i+1)-1$. 
	
	Thus $\mathbb{IG}(R)=G^{*c}(Id(R))=G^c(Id(R))\vee K_{\prod_{i=1}^{n}(n_i+1)-1}$.

	The following two results follows from the above discussion and by Theorem \ref{0-disblowup}, Lemma \ref{cjoin},  Theorem \ref{compthm} and Lemma \ref{ideal}.
	
	\begin{corollary}[{Dodongeh et al. \cite[Theorem 2.1]{inter}}]\label{inters}
		Let $n\geq 2$ be a positive integer and $R\cong \prod_{i=1}^{n}F_{i}$, where $F_{i}$ is a field for every $1\leq i \leq n$. Then $ sdim_{M}(\mathbb{IG}(R))=2^{n}-2^{n-1}-1$.
	\end{corollary}
	
	\begin{corollary}[{Dodongeh et al. \cite[Theorem 3.1]{inter}}]\label{inter}
		Let $R$ ba a commutative ring with identity having finitely many ideals. Suppose that $R \cong \prod_{i=1}^n R_i$, where $R_i$ is a PIR non-field for every $1 \leq i \leq n$ and $n \geq 2$ is a positive integer. Then $ sdim_{M}(\mathbb{IG}(R))=|V(\mathbb{IG}(R))|-2^{n-1}$.	\end{corollary}
	
	\subsection{Component graphs of a vector space}
	\par In \cite{das}, A. Das introduced and analyzed the concept of the nonzero component union graph of a finite-dimensional vector space. 
	
	Let $\mathbb{V}$ be a vector space over a field $\mathbb{F}$, with $\mathcal{B}=\{v_1, \dots, v_n\}$ as its basis and $0$ as the null vector. Any vector $a \in \mathbb{V}$ can be uniquely written as a linear combination in the form $a = a_1v_1 + \dots + a_nv_n$. This expression is referred to as the basic representation of $a$ with respect to the basis $\{v_1, \dots, v_n\}$. The skeleton of $a$ with respect to $\mathcal{B}$ is defined as:
	\[
	S_{\mathcal{B}}(a) = \{v_i \mid a_i \neq 0, \, a = a_1v_1 + \dots + a_nv_n\}.
	\]
	
	\par Angsuman Das \cite{das} defined the \textit{nonzero component graph} $\mathbb{IG(V)}$ with respect to $\mathcal{B}$ as follows: The vertex set of graph  $\mathbb{IG(V)}$ is $\mathbb{V}\setminus \{0\}$ and for any $a, b\in \mathbb{V}\setminus \{0\}$, $a$ is adjacent to  $b$ if and only if $a$ and $b$ share at least one $v_i$ with nonzero coefficient in their basic representation, that is, $a$ and $b$ are adjacent in $\mathbb{IG(V)}$ if and only if $S_{\mathcal{B}}(a)\cap S_{\mathcal{B}}(b)\neq\emptyset$. 
	
	\par In \cite{ncv}, Khandekar et al. essentially gave a relation between the nonzero component graph of a finite-dimensional vector space and the zero-divisor graph of the blow-up of a Boolean lattice. Hence, we have the following result.
	
	\begin{theorem}[Khandekar et al.\cite{ncv}]\label{vector}
		Let $\mathbb{V}$ be an $n$-dimensional vector space over a field $\mathbb{F}$. Then $\mathbb{IG(V)}=G^c({L}^B)$ $\vee K_t$, where $t=|V_{12\dots n}|=(|\mathbb{F}|-1)^n$ and $L^B$ is the blow-up of a Boolean lattice $L\cong \mathbf{2}^n$.
	\end{theorem}
	
	By Lemma \ref{cjoin}, Theorem \ref{compthm} and Theorem \ref{vector}, we have the following result.
	
	\begin{theorem}
		Let $\mathbb{IG(V)}$ be the nonzero component graph of vector spaces with $dim(\mathbb{V})=n\geq 3$. Then $sdim_{M}(\mathbb{IG(V)})=|V(\mathbb{IG(V)})|-2^{n-1}$.
	\end{theorem}
	
	\par\noindent 
	
	\noindent\textbf{Funding:}\\
	First author: None.\\
	Second author: Supported by DST(SERB) under the scheme CRG/2022/002184.
	
	\noindent\textbf{Conflict of interest:} The authors declare that there is no conflict of
	interests regarding the publishing of this paper.
	
	\noindent\textbf{Authorship Contributions :} Both authors contributed to the study on the strong metric dimension of the complement of the zero-divisor graph of a lattice. Both authors read and approved the final version of the manuscript.
	
	\noindent \textbf{Data Availability Statement :} Data sharing does not apply to this article, as no datasets were generated or analyzed during the current study.

\end{document}